\documentclass[12pt,bezier]{article}
\usepackage{tikz}
\usetikzlibrary{shapes.geometric,positioning,calc,decorations.pathreplacing,fit}
\usepackage[margin=0cm]{geometry}
\usepackage{times}
\usepackage{booktabs}
\usepackage{pifont}
\usepackage{floatrow}
\usepackage{caption}
\usepackage{mathrsfs}
\usepackage[fleqn]{amsmath}
\usepackage{breqn}
\usepackage{amsfonts,amsthm,amssymb,mathrsfs,bbding}
\usepackage{txfonts}
\usepackage{graphics,multicol}
\usepackage{graphicx}
\usepackage{color}
\usepackage{amssymb}
\usepackage{amsmath}

\usepackage{caption}
\usepackage{cite}
\usepackage{latexsym,bm}
\usepackage{indentfirst}
\usepackage{color}
\usepackage[colorlinks=true,anchorcolor=blue,filecolor=blue,linkcolor=blue,urlcolor=blue,citecolor=blue]{hyperref}
\usepackage{extarrows}
\usepackage{cite}
\usepackage{latexsym,bm}
\usepackage{mathtools}
\evensidemargin=\oddsidemargin\topmargin=-1.5cm

\newtheorem{thm}{Theorem}[section]
\newtheorem{case}{Case}
\newtheorem{subcase}{Case}[case]
\newtheorem{prob}{Problem}[section]
\newtheorem{corollary}{Corollary}[section]

\newtheorem{lem}{Lemma}[section]

\theoremstyle{definition}

\addtocounter{section}{0}

\begin{document}
\title{Spectral radius and parity $[a,b]$-factors in graphs\footnote{Supported by the National Natural Science Foundation of China
{(No. 12371361)}, Distinguished Youth Foundation of Henan Province {(No. 242300421045)}, and the National Research Foundation of Korea (NRF) grant funded by the Korea government(MSIT) (No. RS-2025-23523950).}}
\author{{\bf Ruifang Liu$^{a}$}, {\bf Ting Xu$^{a}$}, and {\bf Suil O$^{b}$}\thanks{Corresponding author.
E-mail addresses: rfliu@zzu.edu.cn (R. Liu), xuting2001@126.com (T. Xu), suil.o@sunykorea.ac.kr (S. O).}\\
{\footnotesize $^{a}$School of Mathematics and Statistics, Zhengzhou University, Zhengzhou, Henan, China}\\
{\footnotesize $^{b}$Department of Applied Mathematics and Statistics, The State University of New York, Incheon, Korea}}

\date{}

\maketitle
{\flushleft\large\bf Abstract}
Let $a$, $b$, and $n$ be three integers such that $1\leq a \leq b < n$, $a \equiv b$ (mod $2$), and $na$ is even. A parity $[a,b]$-factor of $G$ is a spanning subgraph $H$ such that for each vertex $v \in V(G)$, $a \leq d_H(v) \leq b$ and $d_H(v) \equiv a \equiv b$ (mod $2$). Recently, O [J. Graph Theory 100 (2022) 458-469] proved eigenvalue conditions for a regular graph to have a parity $[a,b]$-factor.

In this paper, we prove a sharp lower bound on the spectral radius for an $n$-vertex graph $G$ to have a parity $[a,b]$-factor as follows: If $G$ is an $n$-vertex connected graph with $\delta(G)\geq a$ and $\rho(G)\geq\rho(G_{n}^{a})$, then $G$ contains a parity $[a,b]$-factor unless $G \cong G_{n}^{a}$, where $2\leq a<b$ and $G_{n}^{a}$ is the graph obtained from $K_{a-1}\vee(K_{n-2a-1}\cup(a+1)K_1)$ by adding a new vertex and adding all possible edges between the added vertex and each vertex in $(a+1)K_1$.

\begin{flushleft}
\end{flushleft}

\section{Introduction}
All graphs considered in this paper are finite, simple, and undirected. Let $G$ be a graph with vertex set $V(G)$ and edge set $E(G)$. Denote by $|V(G)|=n$ and $|E(G)|=m$ the {\it order} and the {\it size} of $G$, respectively. Let $\overline{G}$ be the {\it complement} of $G$. For any two vertex-disjoint graphs $G_1$ and $G_2$, let $G_1\cup G_2$ be the disjoint union of $G_1$ and $G_2$. The {\it join} $G_1\vee G_2$ is the graph obtained from $G_{1}\cup G_{2}$ by adding all possible edges between $V(G_1)$ and $V(G_2)$. For any vertex $v \in V(G)$, let $d_{G}(v)$ be the {\it degree} of $v$ and $N_{G}(v)$ be the neighborhood of $v$ in $G$, respectively. Let $\delta(G)=\min_{v \in V(G)} d_G(v)$. We use $u \sim v$ to represent that $u$ is adjacent to $v$. For any subset $S$ of $V(G)$, denote by $G[S]$ the subgraph of $G$ induced by $S$. Furthermore, let $G-S$ be the induced subgraph $G[V(G)-S]$ for any subset $S$ of $V(G)$. For two vertex disjoint subsets $S, T \subseteq V(G)$, let $|[S,T]|_G$ denote the number of edges between $S$ and $T$.

Let $A(G)$ be the adjacency matrix of $G$, and let $\rho(G)=\lambda_{1}(G)\geq\lambda_{2}(G) \geq \cdots \geq \lambda_{n}(G)$ be its eigenvalues. By the Perron-Frobenius theorem, every connected graph $G$ has a positive unit eigenvector corresponding to $\rho(G)$, which is called the Perron vector of $A(G)$.

Factor theory traces its origins to the pioneering work of Tutte \cite{Tutte} in 1952, which aims to guarantee the existence of a spanning subgraph satisfying specific vertex degree constraints. Let $g$ and $f$ be two integer-valued functions defined on $V(G)$ such that for each $v \in V(G)$, $0 \leq g(v) \leq f(v)$. A {\it $(g,f)$-factor} of $G$ is a spanning subgraph $H$ of $G$ such that for each $v \in V(G)$, $g(v) \leq d_H(v) \leq f(v)$. A {\it parity $(g,f)$-factor} of $G$ is a $(g,f)$-factor $H$ of $G$ such that for each $v \in V(G)$, $d_H(v)\equiv g(v) \equiv f(v)$ (mod $2$). Let $f(S)= \sum_{v \in S}f(v)$ for any $S \subseteq V(G).$  In 1972, Lov\'asz\cite{Lovas} provided a necessary and sufficient condition for a graph to have a parity $(g,f)$-factor.

\begin{thm}[Lov\'{a}sz\cite{Lovas}]\label{thm1.0}
A graph $G$ has a parity $(g,f)$-factor if and only if for any two disjoint subsets $S$, $T$ of $V(G)$,\\
$$f(S)-g(T)+\sum_{x \in T}d_{G-S}(x)-q_{G}(S,T)\geq 0,$$
where $q_{G}(S,T)$ denotes the number of components $Q$ in $G-S-T$ such that $g(V(Q)) + |[V(Q),T]|_G \equiv 1$ $(\rm{mod}~2)$.
\end{thm}

Let $a$ and $b$ be two positive integers with $a \leq b$. An {\it $[a,b]$-factor} of $G$ is a $(g,f)$-factor such that $g(v) \equiv a$ and $f(v) \equiv b$. A {\it parity $[a,b]$-factor} of $G$ is a parity $(g,f)$-factor such that $g(v) \equiv a$ and $f(v) \equiv b$. When $g$ and $f$ are constant, we use $[g,f]$ instead of $(g,f)$ to avoid confusion for the reader. According to Theorem \ref{thm1.0}, one can obtain the following corollary directly.

\begin{corollary}[Lov\'{a}sz\cite{Lovas}]\label{corollary1}
Let $a,b$, and $n$ be three positive integers such that $a\leq b$, $a \equiv b$ $(\rm{mod}~2)$, and $na$ is even. Suppose that $G$ is a graph of order $n$. Then $G$ has a parity $[a,b]$-factor if and only if for any two disjoint subsets $S,T$ of $V(G),$
\begin{eqnarray*}
\eta(S,T)=b|S|-a|T|+\sum_{x \in T}d_{G-S}(x)-q_G(S,T) \geq 0,
\end{eqnarray*}
where $q_G(S,T)$ denotes the number of components $Q$ in $G-S-T$ such that $a|V(Q)| + \allowbreak |[V(Q),T]|_G \equiv 1$ $(\rm{mod}~2)$.
\end{corollary}

For convenience, we call a component $Q$ satisfying $a|V(Q)| +  |[V(Q),T]|_G \equiv 1$ (mod $2$) an $a$-odd component. In Corollary \ref{corollary1}, one can prove that
\begin{eqnarray}\label{eq1}
\eta(S,T)\equiv 0 \text{ (mod $2$)}.
\end{eqnarray}
In fact, let $Q_1, Q_2,\dots, Q_{q'}$ be all components in $G-S-T$. Since $q_{G}(S,T)$ is the number of $a$-odd components in $G-S-T$, we have $q_{G}(S,T) =\sum_{i=1}^{q'}[(a|V(Q_i)|+|[V(Q_i),T]|_G) \text{ (mod 2)}],$ and hence
$$q_{G}(S,T)\equiv (a(n-|S|-|T|)+|[V(G)-S-T,T]|_G)\text{ (mod 2)}.$$
Note that $\eta(S,T) = b(|S|-|T|)+(b-a)|T|+2|E(T)|+|[V(G)-S-T,T]|_G -q_{G}(S,T),$ $a\equiv b$ (mod $2$), and $na$ is even. Then we have
\begin{eqnarray*}
\eta(S,T)\equiv b(|S|-|T|)-a(n-|S|-|T|)\equiv 0 \text{ (mod $2$)}.
\end{eqnarray*}

By Corollary \ref{corollary1} and \eqref{eq1}, we can directly obtain the following corollary.

\begin{corollary}\label{corollary2}
Let $a,b$, and $n$ be three positive integers such that $a \leq b$, $a \equiv b$ $(\rm{mod}~2)$, and $na$ is even. Suppose that $G$ is a graph of order $n$. Then $G$ has no parity $[a,b]$-factor if and only if there exist two disjoint subsets $S,T$ of $V(G)$, $$ \sum_{x \in T} d_{G-S}(x) \leq a|T|-b|S|+ q_{G}(S,T) -2,$$
where $q_{G}(S,T)$ denotes the number of components $Q$ in $G-S-T$ such that $a|V(Q)|\allowbreak + |[V(Q),T]|_G\equiv 1$ $(\rm{mod}~2)$.
\end{corollary}

Since then, motivated by Lov\'{a}sz's $(g,f)$-factor theorem, a lot of researchers have paid much attention to the existence of an $[a,b]$-factor from the perspective of degree condition. Nishimura  \cite{Nishimura} proposed a degree condition for the existence of a $k$-factor, where a {\it $k$-factor} of $G$ is a $[k,k]$-factor of $G$. Subsequently, Li and Cai \cite{Li} extended the above result to an $[a, b]$-factor. In 2018, Liu and Lu \cite{Liu} established a degree condition for the existence of a parity $[a, b]$-factor.

With the development of spectral graph theory, researchers have focused on exploring the connections between eigenvalues and the existence of a parity $[a, b]p1$-factor. In 2010, Lu et al. \cite{Lu} provided an upper bound on $\lambda_3(G)$ for a regular graph to contain an odd $[1, b]$-factor. Kim et al. \cite{Kim} improved the result of \cite{Lu}. O \cite{O} in 2022 established upper bounds for certain eigenvalues to ensure the existence of a parity $[a,b]$-factor in $h$-edge-connected $r$-regular graph. Kim and O \cite{Kim2} proved sharp upper bounds on certain eigenvalues for an $h$-edge-connected graph with given minimum degree to guarantee the existence of a parity $[a,b]$-factor. Wang
et al. \cite{Wang} proposed a tight spectral radius condition for a graph to contain a parity $[a,b]$-factor. Jia et al. \cite{Jia} presented sufficient conditions based on the distance spectral radius and the $Q$-spectral radius for the existence of a parity $[a,b]$-factor in a connected graph, respectively.

Note that $\delta(G)\geq a$ is a trivial necessary condition for a graph $G$ to contain an $[a,b]$-factor. So Hao and Li \cite{Hao} put forward the following interesting and challenging problem.

\begin{prob}\label{p0}
Determine sharp lower bounds on the size or spectral radius of an $n$-vertex graph $G$ with $\delta(G)\geq a$ such that $G$ contains an $[a,b]$-factor.
\end{prob}

Very recently, Tang and Zhang \cite{Tang} solved Problem \ref{p0} for $a=b$. Subsequently, Fan et al. \cite{Fan2} answered Problem \ref{p0} for $a<b$. Let $H_n^{a,b}$ be the graph obtained from $K_a\vee(K_{n-a-b-1}\cup(b+1)K_1)$ by adding $a-1$ edges between one vertex in $(b+1)K_1$ and $a-1$ vertices in $K_{n-a-b-1}$.

\begin{thm}[Fan et al. \cite{Fan2}]\label{thm1.1}
Let $a$ and $b$ be two positive integers with $a<b$, and let $G$ be a connected graph of order $n \geq 2(a+b+2)(b+2)$ with minimum degree $\delta(G)\geq a$. If $\rho(G)\geq \rho(H_n^{a,b})$, then $G$ contains an $[a,b]$-factor unless $G \cong H_n^{a,b}.$
\end{thm}

For $a \equiv b$ $(\rm{mod}~2)$, if a graph $G$ contains a parity $[a,b]$-factor, then it must have an $[a,b]$-factor. Hence one can propose a stronger spectral version of Problem \ref{p0} on parity $[a,b]$-factor.

\begin{prob}\label{p1}
What spectral radius condition suffices to guarantee that a connected graph with minimum degree $\delta(G)\geq a$ has a parity $[a,b]$-factor?
\end{prob}

Tang and Zhang \cite{Tang} solved Problem \ref{p1} for $a=b$. Fan et al. \cite{Fan} solved Problem \ref{p1} for $a=1$.

\begin{thm}[Fan et al. \cite{Fan}]\label{thm1.2}
Suppose that $G$ is a connected graph of even order $n \geq 4b+8$. If
$$\rho(G) \geq \rho(K_1\vee(K_{n-b-2} \cup (b+1)K_1)),$$
then $G$ contains an odd $[1, b]$-factor unless $G \cong K_1\vee(K_{n-b-2} \cup (b+1)K_1)$.
\end{thm}

For general $2\leq a<b$, we in this paper provide a complete solution of Problem \ref{p1}, and present a tight sufficient condition in terms of the spectral radius for a connected graph $G$ with $\delta(G) \geq a$ to contain a parity $[a,b]$-factor. Let $G_{n}^{a}$ be the graph obtained from $K_{a-1} \vee (K_{n-2a-1}\cup(a+1)K_1)$ by adding a new vertex and adding all possible edges between the added vertex and each vertex in $(a+1)K_1$ (See Fig. \ref{fig1}).
\begin{figure}[htbp]
  \centering
  \includegraphics[width=7cm]{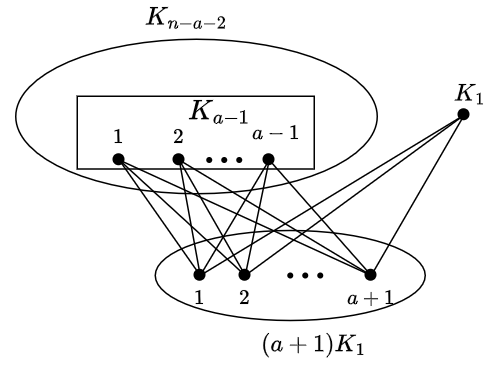}
  \caption{Graph $G_{n}^{a}.$}
  \label{fig1}
\end{figure}

\begin{thm}\label{thm1.3}
Let $a, b$, and $n$ be three positive integers such that $2\leq a < b$, $a \equiv b$ $(\rm{mod}~2)$, and $na$ is even. Suppose that $G$ is an $n$-vertex connected graph such that $n \geq \max\{ 2a^2 + 136a + 264, 2b^2 + 5ab + 7b + 34\}$ and $\delta(G) \geq a$. If
$$\rho(G) \geq \rho(G_{n}^{a}),$$ then $G$ contains a parity $[a,b]$-factor unless $G \cong G_{n}^{a}$.
\end{thm}

We recognize that the graphs holding the bounds in Theorems \ref{thm1.2} and \ref{thm1.3}, are completely different for $1=a\leq b$ and general $2\leq a<b.$

For the reader's convenience, we often omit the subscript ``$G$'' if the context is clear.

\section{Preliminaries}
In this section, we introduce some notation and auxiliary results, which are crucial to the proof of our main result.

\begin{lem}[Brouwer and Haemers \cite{Brouwer2012}]\label{le2.1}
If $H$ is a subgraph of a connected graph $G$, then $\rho(H) \leq \rho(G)$ with equality if and only if $H\cong G$.
\end{lem}

In the following, we present a sharp upper bound on the spectral radius, and this bound was independently proposed by Hong et al. \cite{Hong} and Nikiforov \cite{Nikiforov}.

\begin{lem}[Hong et al.\cite{Hong},  Nikiforov \cite{Nikiforov}]\label{le2.2}
Let $G$ be a graph of order $n$ and size $m$ with positive minimum degree $\delta$. Then
$$\rho(G) \leq \frac{\delta-1}{2} + \sqrt{2m - n\delta + \frac{(\delta+1)^2}{4}},$$
with equality if and only if $G$ is either a $\delta$-regular graph or a bidegree graph in which each vertex is of degree either $\delta$ or $n-1$.
\end{lem}

\begin{lem}[Hong et al.\cite{Hong}, Nikiforov \cite{Nikiforov}]\label{le2.3}
For a simple graph $G$ with order $n$ and size $m$, the function
$$f(x)=\frac{x-1}{2}+\sqrt{2m-nx+\frac{(x+1)^{2}}{4}}$$
is decreasing with respect to $x$ for $0\leq x\leq n-1.$
\end{lem}

The following lemma will be used in the proof of Theorem \ref{thm1.3}.

\begin{lem}[Wu et al.\cite{Wu2005}]\label{le2.4}
Let $G$ be a connected graph with $v_{i}, v_{j}\in V(G)$ and $S\subseteq N(v_{j})\backslash N(v_{i})$. Assume that $G'=G-\{v_{j}v|v\in S\}+\{v_{i}v|v\in S\}$ and $(x_{1},\ldots,x_{n})^{T}$ is the Perron vector of $A(G)$, where $x_{i}$ corresponds to $v_{i}$. If $x_{i}\geq x_{j}$, then $\rho(G')>\rho(G)$.
\end{lem}

Let $M$ be a real $n \times n$ matrix. Assume that $M$ can be written as the following matrix
$$
M = \begin{pmatrix}
M_{1,1} & M_{1,2} & \cdots & M_{1,m} \\
M_{2,1} & M_{2,2} & \cdots & M_{2,m} \\
\vdots & \vdots & \ddots & \vdots \\
M_{m,1} & M_{m,2} & \cdots & M_{m,m}
\end{pmatrix},
$$
whose rows and columns are partitioned into subsets $X_1, X_2, \dots, X_m$ of $\{1,2,\dots,n\}$. The {\it quotient matrix} $R(M)$ of the matrix $M$ (with respect to the given partition) is the $m \times m$ matrix whose entries are the average row sums of the blocks $M_{i,j}$ of $M$. The above partition is called {\it equitable} if each block $M_{i,j}$ of $M$ has constant row (and column) sum.

\begin{lem}[Brouwer and Haemers \cite{Brouwer2012}, Godsil and Royle \cite{Godsil}, Haemers \cite{Haemers}]\label{le2.5}
Let $M$ be a real symmetric matrix and let $R(M)$ be its equitable quotient matrix. Then the eigenvalues of the quotient matrix $R(M)$ are eigenvalues of $M$. Furthermore, if $M$ is nonnegative and irreducible, then the spectral radius of $R(M)$ equals the spectral radius of $M$.
\end{lem}

\begin{lem}[Zhao et al.\cite{Zhao}]\label{le2.6}
Let $n_1 \geq n_2 \geq \cdots \geq n_q$ be positive integers and $n = s + \sum_{i=1}^q n_i$. Then
$$\rho(K_s \vee (K_{n_1} \cup K_{n_2} \cup \cdots \cup K_{n_q})) \leq \rho(K_s \vee (K_{n-s-q+1} \cup (q-1)K_1)),$$ equality holds if and only if when $(n_1, n_2, \dots, n_q) = (n - s - q + 1, 1, \dots, 1)$.
\end{lem}

Next we prove an upper bound on the spectral radius of $K_s \vee (K_{n-b-s-1} \cup (b+1)K_1)$, which plays an important role in the proof of Theorem \ref{thm1.3}.

\begin{lem}\label{le2.7}
If $n, s, b$ are three positive integers with $n \geq \max\{2b,(b+1)s+1\}$, $s\geq 1$, and $b\geq 4$, then we have $\rho(K_s \vee (K_{n-b-s-1} \cup (b+1)K_1)) < n - b - 1$.
\end{lem}

\begin{proof}
Note that $V(K_s \vee (K_{n-b-s-1} \cup (b+1)K_1)) = V(K_s) \cup V(K_{n-b-s-1}) \cup V((b+1)K_1)$ is an equitable partition of $V(K_s \vee (K_{n-b-s-1} \cup (b+1)K_1))$. Then the quotient matrix of $K_s \vee (K_{n-b-s-1} \cup (b+1)K_1)$ can be written as
$$R \triangleq R(A(K_s \vee (K_{n-b-s-1} \cup (b+1)K_1))) = \begin{pmatrix}
    s-1 & n-b-s-1 & b+1 \\
    s & n-b-s-2 & 0 \\
    s & 0 & 0
  \end{pmatrix}.$$
The characteristic polynomial of $R$ is
$$P(R, x) = x^3 - (n-b-3)x^2 - (n+bs+s-b-2)x -b^{2}s+bns-bs^{2}-3bs+ns-s^{2}-2s.$$
Let $\lambda_i(R)$ denote the $i$-th largest eigenvalue of $R$. If $s=1$, then $n\geq 2b$. Combining $b\geq 4$, we have
\begin{eqnarray*}
P(R, n-b-1)&=&n^2 - (2b+1)n + b^2 - b - 2 \geq b^2 - 3b - 2 > 0.
\end{eqnarray*}
If $s\geq 2$, then $n\geq (b+1)s+1$. Combining $b\geq 4$, we have
\begin{eqnarray*}
P(R, n-b-1)&=&n^2-(2b+1)n-bs^2+b^2-bs-s^2+b-s\\
&\geq& (b^2+b)s^2-2(b^2+b)s+b^2-b \geq b^2-b>0.
\end{eqnarray*}
Hence $P(R, n-b-1)>0$. Next we claim that $\lambda_1(R) < n-b-1$. Suppose to the contrary that $\lambda_1(R) \geq n-b-1$. Then we have $\lambda_2(R) > n-b-1$. Note that $P(R, n-b-2) = -bs^2 - s^2 < 0$. This implies that $\lambda_3(R) > n-b-2$, and hence $\lambda_1(R) + \lambda_2(R) + \lambda_3(R) > 3n - 3b - 4$. However, $\lambda_1(R) + \lambda_2(R) + \lambda_3(R) = \operatorname{trace}(R) = n - b - 3$, a contradiction. According to Lemma \ref{le2.5}, $\rho(K_s \vee (K_{n-b-s-1} \cup(b+1)K_1)) < n - b - 1$.
\end{proof}

Now we prove that $G_{n}^{a}$ has no parity $[a,b]$-factor.

\begin{lem}\label{le2.8}
If $a, b,$ and $n$ are three positive integers with $2\leq a < b$, $a \equiv b$ $(\rm{mod}~2)$, and even $na$, then $G_{n}^{a}$ has no parity $[a,b]$-factor.
\end{lem}

\begin{proof}
 Recall that $G_{n}^{a}$ is the graph obtained from $K_{a-1} \vee (K_{n-2a-1}\cup(a+1)K_1)$ by adding a new vertex and adding all possible edges between the added vertex and each vertex in $(a+1)K_1$. Let $V_1 = V(K_1)$, $V_2 = V((a+1)K_1)$, and $V_3 = V(K_{n-a-2})$. Take $S = \emptyset$ and $T = V_2$. Clearly, $q_G(S,T) = 2$. Therefore, one can obtain that
$$b|S| - a|T| + \sum_{x \in T} d_{G-S}(x) - q_G(S,T)= -a(a+1) + (a+1)(a-1) + a+1 - 2 = -2\leq-2,$$
and hence
$$ \sum_{x \in T} d_{G-S}(x) \leq a|T|-b|S|+ q_{G}(S,T) -2.$$
By Corollary \ref{corollary2}, $G_{n}^{a}$ has no parity $[a,b]$-factor.
\end{proof}

\section{Proofs}

Suppose that $a, b$, and $n$ are three positive integers such that $2\leq a < b$, $a \equiv b$ $(\rm{mod}~2)$, $na$ is even, and $n \geq \max\{ 2a^2 + 136a + 264, 2b^2 + 5ab + 7b + 34\}$.
Let $\mathcal{G}_n^{a,b}$ be the family of $n$-vertex connected graphs $G$ such that $\delta(G) \geq a$ and $G$ has no parity $[a,b]$-factor. Suppose that $G^*$ has the maximum spectral radius among graphs in $\mathcal{G}_n^{a,b}$. Let $\boldsymbol{x} = (x_u)_{u\in V(G^*)}$ be the Perron vector of $A(G^*)$. In the following, we write $\rho(G^*)$ as $\rho$ for short. Note that $G^*$ has no parity $[a,b]$-factor. By Corollary \ref{corollary2}, there exist two vertex-disjoint subsets $S,T$ of $V(G^*)$ satisfying
$$\sum_{x \in T} d_{G^*-S}(x) \leq a|T| - b|S| + q_{G^*}(S,T) - 2.$$
We choose $S$ and $T$ such that $|S\cup T|$ is maximized. For convenience, let $ |S|= s$, $ |T|=t$, and $ q_{G^*}(S,T) =q$. It follows that
\begin{eqnarray}\label{eq2}
\sum_{x\in T} d_{G^*-S}(x) \leq at - bs + q - 2.
\end{eqnarray}
Let $Q_1, Q_2, \dots, Q_q$ be the $a$-odd components of $G^*-S-T$, and let $Q_{q+1}, Q_{q+2}, \dots, Q_{q'}$ be the other components of $G^*-S-T$. Furthermore, we denote by $n_i$ the number of vertices of $Q_i$ for each $i\in \{1,2,\dots, q'\}$. Without loss of generality, we assume that $n_1 \geq n_2 \geq \cdots \geq n_q$.
By Lemma \ref{le2.8}, the graph $G_{n}^{a}$ belongs to $\mathcal{G}_n^{a,b}$. Note that $K_{n-a-2}$ is a proper subgraph of $G_{n}^{a}$. Then $\rho\geq\rho(G_{n}^{a}) > \rho(K_{n-a-2}) = n-a-3$. Now we prove an upper bound for $|E(\overline{G^*})|$. For convenience, let $|E(\overline{G^*})|=\overline{m}$ while $|E({G^*})|=m$ and we use $\delta$ instead of $\delta(G^*)$.
\begin{lem}\label{le3.0}
$\overline{m} < (a+2)n - a^2 - \frac{7}{2}a - 3$.
\end{lem}

\begin{proof}
Note that $\delta\geq a$ and $\rho> n-a-3$. By Lemmas \ref{le2.2} and \ref{le2.3}, we obtain that $n-a-3<\rho \leq \frac{a-1}{2} + \sqrt{2m - na + \frac{(a+1)^2}{4}}$, and hence $m > \frac{n^2}{2} - \left(a+\frac{5}{2}\right)n + a^2 + \frac{7}{2}a + 3$, which implies that
\begin{eqnarray*}
\overline{m} &<& \frac{n(n-1)}{2} - \left[\frac{n^2}{2} - \left(a+\frac{5}{2}\right)n + a^2 + \frac{7}{2}a + 3\right] \\
&=& (a+2)n - a^2 - \frac{7}{2}a - 3,
\end{eqnarray*}
which gives the desired result.
\end{proof}

Before proving $G^* \cong G_{n}^{a}$, which is the main goal,
we first prove the following several lemmas, which are very crucial to characterizing the specific structure of $G^*$.

\begin{lem}\label{le3.1}
$d_{G^*}(x)=n-1$ for any $x\in S$ and $Q_i$ is a complete graph for each $i\in\{1,2,\dots,q'\}$.
\end{lem}

\begin{proof}
Let $G'$ be the graph obtained from $G^*$ by adding edges such that $d_{G'}(x) = n-1$ for any $x \in S$ and $Q_i$ is a complete graph for each $i \in \{1,2,\dots,q'\}$. Note that $\delta(G') \geq a$, $q_{G'}(S,T) = q$, and $\sum_{x\in T} d_{G'-S}(x)= \sum_{x\in T} d_{G^*-S}(x)$. By \eqref{eq2}, we have
$$
\sum_{x\in T} d_{G'-S}(x) = \sum_{x\in T} d_{G^*-S}(x) \leq at - bs + q - 2 = at - bs + q_{G'}(S,T) - 2.
$$
By Corollary \ref{corollary2}, $G'$ has no parity $[a,b]$-factor, and hence $G' \in \mathcal{G}_n^{a,b}$. According to the maximality of $\rho$, we have $\rho(G') \leq \rho$. Note that $G^*$ is a spanning subgraph of $G'$. By Lemma \ref{le2.1}, we can deduce that $G^* \cong G'$, which gives the desired result.
\end{proof}

\begin{lem}\label{le3.2}
$S \cup T \neq \emptyset$.
\end{lem}

\begin{proof}
Assume that $S \cup T = \emptyset$. Combining \eqref{eq2} and $\sum_{x \in T} d_{G^*-S}(x)\geq 0$, we can obtain that $q \geq 2$, which contradicts that $G^*$ is connected.
\end{proof}

Furthermore, we can derive an upper bound on the number of vertices in $T$.
\begin{lem}\label{le3.3}
$t \leq \frac{1}{8}n$.
\end{lem}

 \begin{proof}
Assume to the contrary that $t > \frac{1}{8}n$. We consider the non-edges in $T$, together with the non-edges among $a$-odd components $Q_1, Q_2,\dots, Q_q$. By (\ref{eq2}), $t > \frac{1}{8}n$, and $n \geq 2a^2 + 136a + 164$, we have

\begin{eqnarray*}
\overline{m}&\geq&\frac{t(t-1)}{2} - \frac{1}{2}\sum_{x\in T} d_{G^*-S}(x) + q - 1 \geq \frac{t(t-1)}{2} - \frac{1}{2}(at - bs + q - 2) + q - 1 \\
&=& \frac{t(t-1)}{2} - \frac{1}{2}at + \frac{1}{2}bs + \frac{1}{2}q>\frac{t(t-1)}{2} - \frac{1}{2}at - \frac{1}{2} \\
&>& \frac{1}{128}n^2 - \left(\frac{a}{16} + \frac{1}{16}\right)n - \frac{1}{2}> (a+2)n - a^2 - \frac{7}{2}a - 2,
\end{eqnarray*}
which contradicts Lemma \ref{le3.0}.
\end{proof}

Now we can estimate the vertex degree of $a$-odd components in $G^*-S.$
\begin{lem}\label{le3.4}
$d_{G^*-S}(x) \geq a+1$ for any $x \in V(Q_1 \cup Q_2 \cup \cdots \cup Q_q)$.
\end{lem}

\begin{proof}
Assume to the contrary that there exists some vertex $x \in V(Q_1 \cup Q_2 \cup \cdots \cup Q_q)$ such that $d_{G^*-S}(x) \leq a$.
Let $T' = T \cup \{x\}$. Then we can obtain that
\begin{eqnarray*}
\eta(S,T') &=& bs - a|T'| + \sum_{u\in T'} d_{G^*-S}(u) - q_{G^*}(S,T') \\
&=& bs - a(t+1) + \sum_{u\in T} d_{G^*-S}(u) + d_{G^*-S}(x) - q_{G^*}(S,T') \\
&\leq& bs - at + \sum_{u\in T} d_{G^*-S}(u) - \left(q - 1\right) \\
&\leq& -1.
\end{eqnarray*}
By \eqref{eq1}, we have
$\eta(S,T') \leq -2$, which implies that $$\sum_{x\in T'} d_{G^*-S}(x) \leq a|T'| - bs + q_{G^*}(S,T') - 2.$$ Note that $|S\cup T'|=|S\cup T|+1$. This contradicts the maximality of $|S\cup T|$.
\end{proof}

Next, our goal is to determine that all components in $G^*-S-T$ are $a$-odd components.

\begin{lem}\label{le3.5}
$q\geq 2$.
\end{lem}

\begin{proof}
Suppose that $0 \leq q \leq 1$. Recall that $\delta \geq a$. Combining \eqref{eq2}, we have
$$at - st\leq \sum_{x\in T} d_{G^*-S}(x) \leq at - bs - 1,$$
and hence $s \geq 1$ and $t \geq b+ \frac{1}{s}$, which implies that $t \geq b+1$. Note that $\sum_{x\in T} d_{G^*-S}(x) \geq 0$ and $a<b$. Then $s \leq \frac{a}{b}t - \frac{1}{b} < t$. We consider the non-edges between $T$ and $V(G^*)-S-T$. By $ 1\leq s < t$, $b+1 \leq t \leq \frac{n}{8} $, and $n \geq 2b^2 + 5ab + 7b + 34$, we have
\begin{eqnarray*}
\overline{m} &\geq& |T| |V(G^*)-S-T| - \sum_{x\in T} d_{G^*-S}(x) \\
&\geq& t(n-s-t) - at + bs + 1 > t(n-2t - a) + b+1 \\
&\geq& (b+1)(n - 2b - 2 - a) + b+1 > (a+2)n - a^2 - \frac{7}{2}a - 2,
\end{eqnarray*}
contradicting Lemma \ref{le3.0}.
\end{proof}

\begin{lem}\label{le3.6}
All components of $G^* - S - T$ are $a$-odd components.
\end{lem}

\begin{proof}
Assume to the contrary that there exists a component $Q$, which is not an $a$-odd component of $G^* - S - T$. Then we have $a|V(Q)|+ |[V(Q),T]|_{G{^*}} \equiv 0$ (mod $2$). By Lemma \ref{le3.5}, $Q_1$ must exist. Construct a new graph $G'$ from $G^*$ by adding edges between $Q_1$ and $Q$ such that $G'[V(Q_1 \cup Q)] = K_{n_{1}'}$, where $n_{1}' = n_1 + |V(Q)|$. Since $an_{1}'+ |[V(Q_1\cup Q),T]|_{G'} \equiv 1$ $(\rm{mod}~2)$, we can deduce that $K_{n_{1}'}$ is still an $a$-odd component. Clearly, $\delta(G') \geq a$, $q_{G'}(S,T) = q$, and $ \sum_{x\in T} d_{G'-S}(x) = \sum_{x\in T} d_{G^*-S}(x).$ It follows that
$$\sum_{x\in T} d_{G'-S}(x) = \sum_{x\in T} d_{G^*-S}(x) \leq at - bs + q - 2 = at - bs + q_{G'}(S,T) - 2.
$$
By Corollary \ref{corollary2}, $G'$ has no parity $[a,b]$-factor. Hence $G' \in \mathcal{G}_n^{a,b}$. Note that $G^*$ is a proper subgraph of $G'$. By Lemma \ref{le2.1}, we have $\rho < \rho(G')$, contrary to the maximality of $\rho$.
\end{proof}

Furthermore, we determine the exact number of $a$-odd components.
\begin{lem}\label{le3.7}
$q = 2$.
\end{lem}

\begin{proof}
By Lemma \ref{le3.5}, we know that $q\geq 2$. Suppose to the contrary that $q\geq3$.
Next we divide the proof into the following three cases.
\begin{case}\label{subcase2.1}
$t = 0$.
\end{case}

It follows from \eqref{eq2} that $\sum_{x \in T} d_{G^*-S}(x)\leq -bs+q-2.$ By Lemma \ref{le3.2}, $s \geq 1$. Combining $-bs+q-2 \geq 0 $ and $s \geq 1$, we have $q \geq bs + 2 \geq b + 2$. By $bs+2 \leq q \leq n - s$, we can deduce that $n\geq (b+1)s+2>(b+1)s+1$.
Note that $n \geq 2b^2 + 5ab + 7b + 34>2b$. Then we have $n\geq \max\{2b,(b+1)s+1\}$. According to Lemmas \ref{le3.1} and \ref{le3.6}, we can deduce that $G^*\cong K_s\vee(K_{n_1}\cup K_{n_2}\cup \cdots\cup K_{n_q})$. By $q\geq b+2$, $b\geq a+2$, Lemmas \ref{le2.6} and \ref{le2.7}, we have
\begin{eqnarray*}
\rho&\leq&\rho(K_s\vee(K_{n-s-q+1}\cup (q-1)K_1))\leq\rho(K_s\vee(K_{n-s-b-1}\cup(b+1)K_1))\\
&<&n-b-1\leq n-a-3,
\end{eqnarray*}
which contradicts $\rho>n-a-3$.

\begin{case}\label{subcase2.2}
$t=1$.
\end{case}

Suppose that $T = \{u_0\}$. In this case, we have
\begin{eqnarray}\label{eq3}
\sum_{x\in T} d_{G^*-S}(x)=d_{G^*-S}(u_0) \leq a-bs+q-2.
\end{eqnarray}
Next we divide the discussion according to the value of $s$.

\begin{subcase}
$s=0.$
\end{subcase}

Since $G^*$ is connected, $|[T,V(Q_i)]|_{G^*}\geq 1$ for any $1\leq i\leq q$. For $q=3$, define $G' = G^*$ and $V(Q_{i}')=V(Q_i)$ for $1\leq i\leq 3$. For $q\geq 4$, we construct a new graph $G'$ from $G^*$ by deleting one edge between $T$ and $V(Q_i)$ for each $4\leq i\leq q$ and adding edges between $V(Q_i)$ and $V(Q_1)$ such that $G'[V(Q_1) \cup V(Q_4)\cup \cdots \cup V(Q_q)] = K_{n_{1}'}$, where $n_{1}' = n_1 + n_4 + \dots + n_q$. Clearly, $|E(\overline{G'})|\leq \overline{m}$. Note that $an_1'+|[T,V(Q_1')]|_{G'} \equiv 1$ (mod $2$). Then $G'-S-T$ has three $a$-odd components $Q_1', Q_2', Q_3'$ satisfying $V(Q_1') = V(Q_1) \cup V(Q_4) \cdots \cup V(Q_q)$, $V(Q_2') = V(Q_2)$, and $V(Q_3') = V(Q_3)$. Combining \eqref{eq3}, we have
\begin{eqnarray}\label{eq4}
\sum_{x\in T} d_{G'-S}(x) = \sum_{x\in T} d_{G^*-S}(x)-(q-3)\leq a + q - 2 - (q-3)= a+ 1.
\end{eqnarray}
Let $H = T \cup V(Q_2') \cup V(Q_3')$. By Lemma \ref{le3.4}, $|H| \geq a+3$. Next we assert that $|H| \leq \frac{1}{2}n - b - 1$. Otherwise, suppose that $|H| \geq \frac{1}{2}n - b$. Since $t=1$, we have $|V(Q_2') \cup V(Q_3')| \geq \frac{1}{2}n - b - 1$. Note that $|V(Q_1')| \geq |V(Q_2')| \geq |V(Q_3')|$. Then we can obtain that $|V(Q_1')| \geq \frac{n-1}{3}$. Next consider the non-edges of $G'$ between $V(Q_1')$ and $V(Q_2') \cup V(Q_3')$. Since $n \geq 2b^2 + 5ab + 7b + 34$, we have
\begin{eqnarray*}
\overline{m} &\geq& |E(\overline{G'})|\geq \frac{n-1}{3}\left(\frac{1}{2}n - b - 1\right) = \frac{1}{6}n^2 - \left(\frac{b}{3} + \frac{1}{2}\right)n + \frac{b+1}{3} \\
&>& (a+2)n - a^2 - \frac{7}{2}a - 2,
\end{eqnarray*}
contrary to Lemma \ref{le3.0}. Hence $a+3 \leq |H| \leq \frac{1}{2}n - b - 1$.
Now we consider the non-edges of $G'$ between $H$ and $V(Q_1')$, together with one non-edge between $V(Q_2')$ and $V(Q_3')$. By \eqref{eq4} and $n \geq 2a^2 + 136a + 164$, one can obtain that
\begin{eqnarray*}
\overline{m} &\geq& |E(\overline{G'})|\geq |H| |V(Q_1')| + 1 - \sum_{x\in T} d_{G'-S}(x)\geq |H|  (n - |H|) - a\\
&\geq& (a+3)(n - a - 3)-a > (a+2)n - a^2 - \frac{7}{2}a - 2,
\end{eqnarray*}
which contradicts Lemma \ref{le3.0}.

\begin{subcase}
$s \geq 1.$
\end{subcase}

According to $\delta \geq a$ and \eqref{eq3}, $a - s \leq d_{G^*-S}(u_0) \leq a - bs + q - 2$. Hence $q \geq (b-1)s + 2 \geq b+1$. Combining $a - bs + q - 2\geq 0$ and $q \leq n - s - 1$, we can obtain that $s\leq \frac{n}{b+1}+\frac{a-3}{b+1} < \frac{n}{b+1} + 1$. By $t=1$ and  Lemma \ref{le3.4} , one can deduce that $n_i \geq a+1$ for $1\leq i\leq q$. Now consider the non-edges among the $Q_1, Q_2, \dots, Q_q$. By $q \geq b+1$, $s < \frac{n}{b+1} + 1$, $b\geq a+2\geq 4$, and $n \geq 2b^2 + 5ab + 7b + 34$, we have
\begin{eqnarray*}
\overline{m} &\geq& \frac{1}{2}\sum_{i=1}^q n_i(n - s - 1 - n_i)\geq \frac{1}{2}(a+1) (q-1)(n - s - 1)\\
&>&  \frac{(a+1)b^2}{2(b+1)} n - ab - b \geq \frac{3b^2}{2(b+1)} n - ab - b \geq   \frac{b(b+2)}{b+1} n - ab - b \\
&\geq& (a+2)n + \frac{b}{b+1} n - ab - b>(a+2)n - a^2 - \frac{7}{2}a - 2,
\end{eqnarray*}
contradicting Lemma \ref{le3.0}.

\begin{case}\label{subcase2.3}
$t \geq 2$.
\end{case}

\begin{subcase}
$s=0$.
\end{subcase}

It follows from \eqref{eq2} that
\begin{eqnarray}\label{eq5}
\sum_{x \in T}d_{G^*-S}(x) \leq at+q-2
\end{eqnarray}
Since $G^*$ is connected, $|[T,V(Q_i)]|_{G^*}\geq 1$ for any $1\leq i\leq q$. For $q=3$, we define $G_{1}=G^*$ and $V(Q_i')=V(Q_i)$. For $q\geq4$, we construct a new graph $G_{1}$ from $G$ by deleting one edge between $T$ and $V(Q_i)$ for any $4\leq i\leq q$ and connecting $V(Q_1)$ and $V(Q_i)$ such that $G_{1}[V(Q_1) \cup V(Q_4)\cup \cdots \cup V(Q_q)] = K_{n_{1}'}$, where $n_{1}' = n_1 + n_4 + \dots + n_q$. Clearly, $|E(\overline{G_1})| \leq \overline{m}$. Observe that $G_{1}-S-T$ has three $a$-odd components $Q_1', Q_2'$, $Q_3'$ satisfying $V(Q_1') = V(Q_1) \cup V(Q_4)\cup \cdots \cup V(Q_q)$, $V(Q_2') = V(Q_2)$, and $V(Q_3') = V(Q_3)$. By \eqref{eq5}, we can obtain that
\begin{eqnarray}\label{eq6}
\sum_{x\in T} d_{G_{1}-S}(x)&=&\sum_{x\in T} d_{G^*-S}(x) -(q-3)\leq at + q - 2-(q-3) = at + 1.
\end{eqnarray}
Now we claim that $t \leq a+2$. Suppose to the contrary that $t \geq a+3$. Focus on the non-edges of $G_{1}$ between $T$ and $V(G_1)-T$. According to \eqref{eq6}, $a+3 \leq t \leq \frac{1}{8}n$, and $n \geq 2a^2 + 136a + 164$, we deduce that
\begin{eqnarray*}
\overline{m} &\geq& |E(\overline{G_1})|\geq |T||V(G_1)-T|-\sum_{x\in T}d_{G_{1}-S}(x)\geq t(n-t) - at - 1 \\
&\geq& (a+3)(n - a - 3) - a(a+3) - 1 >(a+2)n - a^2 - \frac{7}{2}a - 2,
\end{eqnarray*}
contrary to Lemma \ref{le3.0}. Hence $t \leq a+2$. Let $H_{1} = T \cup V(Q_2') \cup V(Q_3')$. By Lemma \ref{le3.4}, we have $|H_{1}| \geq a+3$. Next we assert that $|H_{1}| \leq \frac{1}{2}n - b - 1$. In fact, suppose that $|H_{1}| \geq \frac{1}{2}n - b$. Since $t \leq a+2$, $|V(Q_2')| + |V(Q_3')| \geq \frac{1}{2}n - b- a - 2 > \frac{1}{2}n - 2b - 2$. Observe that $|V(Q_1')| \geq |V(Q_2')| \geq |V(Q_3')|$. By $t \leq a+2$ and $s=0$, we have $|V(Q_1')| \geq \frac{n-s-t}{3}\geq \frac{n - a - 2}{3}$. In the following, we consider the non-edges between $V(Q_1')$ and $V(Q_2') \cup V(Q_3')$. Together with $n \geq 2b^2 + 5ab + 7b + 34$, we have
\begin{eqnarray*}
\overline{m} &\geq& |E(\overline{G_1})| > \frac{1}{3}(n - a - 2)\left(\frac{1}{2}n - 2b - 2\right)  \\
&=& \frac{n^2}{6} - \left(\frac{a}{6} + \frac{2b}{3} + 1\right)n + \frac{2}{3}(a+2)(b+1)  \\
&>& (a+2)n - a^2 - \frac{7}{2}a - 2,
\end{eqnarray*}
which contradicts Lemma \ref{le3.0}. Hence $|H_{1}| \leq \frac{1}{2}n - b - 1$. Consider the number of the non-edges between $H_{1}$ and $V(Q_1')$, together with one non-edge between $V(Q_2')$ and $V(Q_3')$. Combining \eqref{eq6}, $t \leq |H_{1}|$, $a+3 \leq |H_{1}| \leq \frac{1}{2}n - b - 1$, and $n \geq 2a^2 + 136a + 164$, we have
\begin{eqnarray*}
\overline{m} &\geq& |E(\overline{G_1})|\geq|H_{1}||V(Q_1')| + 1 - \sum_{x\in T} d_{G_{1}-S}(x) \\
&\geq& |H_{1}|(n - |H_{1}|) - at\geq|H_{1}|(n-a-|H_{1}|)\\
&\geq&(a+3)(n - 2a - 3)>(a+2)n-a^2-\frac{7}{2}a-2,
\end{eqnarray*}
contradicting Lemma \ref{le3.0}.

\begin{subcase}
$s \geq 1$.
\end{subcase}

For $3 \leq q \leq b+1$, we define $G_{2} = G^*$. By (\ref{eq2}), we have
$$\sum_{x\in T} d_{G_2-S}(x) \leq at - bs + b - 1.$$
For convenience, let $V(Q_j')=V(Q_j)(1\leq j\leq q)$. For $q \geq b+2$, let $r$ be the number of the $a$-odd components with $n_i \geq a+2$ for $1\leq i\leq q$. For $0 \leq r \leq b+1$, we can obtain that $n_i \leq a+1(b+2 \leq i\leq q)$. Combining this and Lemma \ref{le3.4}, we have $|[T,V(Q_i)]|_{G^*}\geq1$ for $b+2 \leq i\leq q$. We construct a new graph $G_{2}$ from $G^*$ by deleting one edge between $T$ and $Q_i$ for any $b+2 \leq i\leq q$ and attaching $V(Q_i)$ to $V(Q_1)$ such that $G_{2}[V(Q_1) \cup V(Q_{b+2})\cup\cdots\cup V(Q_q)] = K_{n_{1}'}$, where $n_{1}' = n_1 + n_{b+2} + \dots + n_q$. Clearly, $|E(\overline{G_2})| \leq \overline{m}$. In this case, observe that $G_{2}-S-T$ has $b+1$ $a$-odd components $Q_i'$ ($1\leq i\leq b+1$) satisfying $V(Q_1')=V(Q_1) \cup V(Q_{b+2})\cup\cdots\cup V(Q_q)$, $V(Q_j') = V(Q_j)$ ($2\leq j\leq b+1$). By \eqref{eq2}, we have
\begin{eqnarray}\label{eq7}
\sum_{x\in T} d_{G_{2}-S}(x) &=& \sum_{x\in T} d_{G^*-S}(x)-(q - b - 1)  \leq at-bs+b-1.
\end{eqnarray}
According to $\sum_{x\in T} d_{G_{2}-S}(x) \geq 0$, we can obtain that $s \leq \frac{a}{b}t+\frac{b-1}{b}<t+1$. Now we claim that $t \leq a+2$. Suppose to the contrary that $t \geq a+3$. Consider the non-edges of $G_{2}$ between $T$ and $V(G_{2})-S-T$. By \eqref{eq7}, $1\leq s< t+1$, $a+3 \leq t \leq \frac{1}{8}n$, and $n \geq 2a^2 + 136a + 164$,
\begin{eqnarray*}
\overline{m} &\geq& |E(\overline{G_2})| \geq |T||V(G_{2})-S-T|-\sum_{x \in T}d_{G_{2}-S}(x)\\
&\geq& t(n - s - t) - at + bs - b + 1 > t(n - 2t - 1) - at + 1  \\
&\geq& (a+3)n - 3a^2 - 16a - 20 > (a+2)n - a^2 - \frac{7}{2}a - 2,
\end{eqnarray*}
which contradicts Lemma \ref{le3.0}. Hence $t \leq a+2$. Then $s < t+1 \leq a+3$. Let $H_{2}= V(G_{2})-V(Q_1')-S$. According to Lemma \ref{le3.2}, we have $|H_{2}| \geq a+3$. If $|H_{2}| \geq \frac{1}{2}n - a - 1$, then by $t \leq a+2$, we can deduce that $|V(G_{2})-S-T-V(Q_1')| \geq \frac{1}{2}n - 2a - 3$.
Next our goal is to prove that $|V(Q_1')|>\frac{n - 2a - 5}{b+1}$ in two cases. When $3\leq q\leq b+1$, we have $|V(Q_1')|\geq |V(Q_2')|\geq\cdots \geq |V(Q_{q}')|$, which implies that $|V(Q_1')| \geq \frac{n- s - t}{q} \geq\frac{n - s - t}{b+1}>\frac{n - 2a - 5}{b+1}$. When $q \geq b+2$ and $0\leq r\leq b+1$, we have $|V(Q_1')| \geq |V(Q_2')| \geq \cdots \geq |V(Q_{b+1}')|$, implying $|V(Q_1')| \geq \frac{n - s - t}{b+1} >\frac{n - 2a - 5}{b+1}$. Consider the non-edges of $G_{2}$ between $V(Q_1')$ and $V(G_{2})-S-T-V(Q_1')$. By $n\geq2b^2+5ab+7b+34$, we have
\begin{eqnarray*}
\overline{m} &\geq& |E(\overline{G_2})|> \frac{1}{b+1}(n - 2a - 5)\left(\frac{1}{2}n - 2a - 3\right)   \\
&=& \frac{1}{2(b+1)} n^2 - \left(\frac{2a}{b+1} + \frac{11}{2(b+1)}\right)n + \frac{1}{b+1}(4a^2 + 16a + 15) \\
&>& (a+2)n - a^2 - \frac{7}{2}a - 2,
\end{eqnarray*}
contrary to Lemma \ref{le3.0}. Hence $|H_{2}|\leq\frac{1}{2}n - a - 1$. Calculate the non-edges of $G_{2}$ between $H_{2}$ and $V(Q_1')$. Combining \eqref{eq7}, $1\leq s<a+3$, $t\leq a+2$, $a+3\leq|H_{2}|\leq\frac{1}{2}n - a - 1$, and $n \geq 2a^2 + 136a + 164$, we obtain that
\begin{eqnarray*}
\overline{m} &\geq& |E(\overline{G_2})| \geq |H_{2}|  |V(Q_1')| - \sum_{x\in T} d_{G_{2}-S}(x) \\
&\geq& |H_{2}|  (n - s - |H_{2}|) - at + bs - b + 1 >|H_{2}|  (n - a - 3 - |H_{2}|) - a(a+2) + 1  \\
&\geq& (a+3)(n - 2a - 6) - a^2 - 2a + 1  > (a+2)n - a^2 - \frac{7}{2}a - 2,
\end{eqnarray*}
which contradicts Lemma \ref{le3.0}.

For $r\geq b+2$, by $at-bs+q-2 \geq 0$, $bs-at+2\leq q \leq n-s-t$, and Lemma \ref{le3.3}, we deduce that
$$
s \leq \frac{a-1}{b+1} t + \frac{n-2}{b+1} < t + \frac{1}{4}n \leq \frac{3}{8}n.
$$
Now it suffices to consider the non-edges of $G^*$ among the $a$-odd components $Q_{1},Q_{2},\dots,\allowbreak Q_{b+2}$. Recall that $n_i \geq a+2$ for any $1\leq i\leq b+2$. According to $s< \frac{3}{8}n $, Lemma \ref{le3.3}, and $n \geq 2a^2 + 136a + 164$, we have
\begin{eqnarray*}
\overline{m} &\geq& \frac{1}{2}\sum_{i=1}^{b+2} n_i  (n - s - t - n_i) \geq \frac{a+2}{2}  \left[(b+2)(n - s - t) - \sum_{i=1}^{b+2} n_i\right] \\
&\geq&2(b+1)(n - s - t)> (b+1)n  \geq (a+3)n \\
&>& (a+2)n - a^2 - \frac{7}{2}a - 2,
\end{eqnarray*}
which contradicts Lemma \ref{le3.0}.
\end{proof}

According to \eqref{eq2} and Lemma \ref{le3.7}, we have
\begin{eqnarray}\label{eq8}
\sum_{x\in T} d_{G^*-S}(x) \leq at - bs.
\end{eqnarray}
Let $U = T \cup V(Q_2)$. By Lemma \ref{le3.4}, we can obtain that $|U| \geq a+2$. Next we claim that $|U| \leq \frac{1}{2}n - 2b - 1$. Assume that $|U| \geq \frac{1}{2}n - 2b$. Note that $\sum_{x\in T} d_{G^*-S}(x) \geq 0$. By (\ref{eq8}), we can obtain that $s \leq \frac{a}{b}t < t$. Combining $n_1 \geq n_2$ and Lemma \ref{le3.3}, we have $n_1 \geq \frac{n - s - t}{2} > \frac{3}{8}n$. Furthermore, $n_2=|U|-t\geq\frac{3}{8}n-2b$.
Now we consider the non-edges between $V(Q_1)$ and $V(Q_2)$. By $n\geq 2b^2+5ab+7b+34$, we have
$$
\overline{m} > \frac{3}{8}n \left(\frac{3}{8}n - 2b\right) = \frac{9}{64}n^2 - \frac{3}{4}bn > (a+2)n - a^2 - \frac{7}{2}a - 2,
$$
contradicting Lemma \ref{le3.0}. Hence $|U| \leq \frac{1}{2}n - 2b - 1$. Moreover, we prove that $S=\emptyset$.

\setcounter{case}{0}

\begin{lem}\label{le3.8}
$s=0$.
\end{lem}

\begin{proof}
Assume that $s\geq1.$
Note that $\sum_{x \in T} d_{G^*-S}(x) \geq at - st$. By (\ref{eq8}), we have $t\geq b\geq a+2$. Hence $|U|\geq a+3$. Recall that $s<t$. We consider the number of non-edges between $U$ and $V(Q_1)$. By \eqref{eq8}, $1\leq s<t<|U|$, $a+3 \leq |U| \leq \frac{1}{2}n-2b-1$, and $n \geq 2a^2+136a+264$, we can obtain that
\begin{eqnarray*}
\overline{m}&\geq& |U|(n-s-|U|)-\sum_{x \in T}d_{G^*-S}(x) \geq|U|(n-s-|U|)-at+bs\\
&>&|U|(n-|U|-t)-at+b>|U|(n-2|U|-a)+b\\
&\geq&(a+3)(n-3a-6)+b>(a+2)n-a^2-\frac{7}{2}a-2,
\end{eqnarray*}
contrary to Lemma \ref{le3.0}.
\end{proof}

By Lemma \ref{le3.8}, $s=0.$ Combining Lemma \ref{le3.2}, we have $t\geq 1$. Recall that $\delta \geq a$. It follows from \eqref{eq8} that
\begin{eqnarray*}
at\leq  \sum_{x\in T} d_{G^*}(x)=\sum_{x\in T} d_{G^*-S}(x)\leq at,
\end{eqnarray*}
which implies that $d_{G^*}(x) = a$ for any $x\in T$.
Now we claim that $|U|=a+2$.
If $|U|\geq a+3$, we consider the non-edges between $U$ and $V(Q_1)$. By $|U|\geq t\geq 1$, $a+3 \leq |U| \leq \frac{1}{2}n - 2b - 1$, and $n \geq 2a^2 + 136a + 164$, we have
\begin{eqnarray*}
\overline{m} &\geq& |U|(n - |U|) - \sum_{x\in T} d_{G^*}(x) =|U|  (n - |U|) - at \geq|U|  (n - a - |U|) \\
&\geq& (a+3)(n - 2a - 3) > (a+2)n - a^2 - \frac{7}{2}a - 2,
\end{eqnarray*}
which contradicts Lemma \ref{le3.0}. Hence $|U|=a+2$.

\vspace{3mm}
Next we determine the number of vertices in $T$.

\begin{lem}\label{le3.9}
$t=a+1.$
\end{lem}

\begin{proof}
By Lemma \ref{le3.7}, $n_2\geq1 $. Combining $|U|=a+2$, $t\leq a+1$. Assume to the contrary that $t\leq a.$ Then $n_2 \geq 2$. By Lemma \ref{le3.4},  $|[V(Q_2),T]|_{G^*} = tn_2$. Since $Q_2$ is an $a$-odd component of $G^*-S-T$, $an_2+tn_2$ is odd. Then $n_2$ is odd, and hence $n_2 \geq 3$ and $t \leq a-1$. Next we assert that $t\geq 3$ and $n_2 \leq a-1$. If $t=1$, then $|[V(Q_2),T]|_{G^*} = n_2=a+1$, and hence $d_{G^*}(x)> a+1$ for any $x \in T$,  contradicting $d_{G^*}(x)=a$. If $t=2$, then $n_2=a$. Since $d_{G^*}(x)=a$ for any $x \in T$, $|[V(Q_1),T]|_{G^*}=0$, which contradicts that $G^*$ is connected. Hence $t\geq 3$ and $n_2 \leq a-1$.

Let $T=\{u_1,u_2,\dots,u_t\}$, $V(Q_1)=\{v_1,v_2,\dots,v_{n-a-2}\}$, and $V(Q_2)=\{w_1,w_2,\dots,w_{n_2}\}$. Without loss of generality, assume that $x_{u_1}\geq x_{u_2}\geq \cdots \geq x_{u_t}$ and $x_{v_1}\geq x_{v_2}\geq \cdots \geq x_{v_{n-a-2}}$. Note that $|[V(Q_2),T]|_{G^*} = tn_2$. Then $x_{w_1}= x_{w_2} = \cdots = x_{w_{n_2}}$. By $A(G^*)\boldsymbol x=\rho \boldsymbol x$, we have $$\rho x_{w_1}=\sum_{i=1}^{t}x_{u_i}+(n_2-1)x_{w_1}\leq tx_{u_1}+(n_2-1)x_{w_1}.$$ Combining $\rho > n-a-3$, $t\leq a-1$, and $n\geq2a^2+136a+164$, we have
\begin{eqnarray}\label{eq9}
x_{u_1}> \frac{t}{\rho-n_2+1}x_{u_1}\geq x_{w_1}.
\end{eqnarray}
Let $d_{T}(u_1)=r$. By \eqref{eq9} and $d_{G^*}(x) = a$ for any $x\in T$, we have
\begin{eqnarray*}
\rho x_{u_1}&=&\sum_{u \in T, u\sim u_1} x_u +\sum_{v \in V(Q_1),v \sim u_1}x_v+n_2x_{w_1}<(r+n_2)x_{u_1}+\sum_{1\leq i\leq a-r-n_2} x_{v_i},
\end{eqnarray*}
and hence $x_{u_1}<\frac{1}{\rho-r-n_2}\sum_{1\leq i\leq a-r-n_2}x_{v_i}$. Since $x_{u_1}>0$, we have $a-r-n_2\geq 1$, which implies that $r+n_2\leq a-1$. Note that
\begin{eqnarray*}
\rho x_{v_{n-a-2}}\geq \sum_{v \in V(Q_1), v\sim v_{n-a-2}}x_v\geq \sum_{1\leq i\leq a-r-n_2}x_{v_i}+(n-2a+r+n_2-3)x_{v_{n-a-2}}.
\end{eqnarray*}
Then $x_{v_{n-a-2}}\geq \frac{1}{\rho-(n-2a-3)-r-n_2}\sum_{1\leq i\leq a-r-n_2}x_{v_i}$. By $\rho>n-a-3$, $n \geq 2a^2 + 136a + 164$, and $r+n_2\leq a-1$, we have $1<\rho-(n-2a-3)-r-n_2<\rho-r-n_2$, and hence
\begin{eqnarray}\label{eq10}
x_{v_{n-a-2}}\geq \frac{1}{\rho-(n-2a-3)-r-n_2}\sum_{1\leq i\leq a-r-n_2}x_{v_i} >\frac{1}{\rho-r-n_2}\sum_{1\leq i\leq a-r-n_2}x_{v_i}>x_{u_1}.
\end{eqnarray}

Next we claim that $E(G^*[T])= \emptyset$. Otherwise, there exists an edge $uu'\in E(G^*[T])$. Since $|N_{G^*}(u)\cup N_{G^*}(u')|\leq 2a <|V(Q_1)|$, there exists some vertex $v\in V(Q_1)$ such that $vu, vu' \notin E(G^*)$. We construct a new graph $G'$ from $G^*$ by deleting the edge $uu'$ and adding edges $vu$ and $vu'$.
One can check that \eqref{eq2} still holds for $G'$, $\delta(G')\geq a$ and $G'$ is connected. So $G'\in \mathcal{G}_n^{a,b}$. By \eqref{eq10} and Lemma \ref{le2.4}, we have $\rho(G')>\rho,$ which contradicts the maximality of $G^*$. Hence $E(G^*[T])= \emptyset$.

Now we prove that $u$ is adjacent to $v_ 1, v_2, \ldots, v_{a-n_2}$ for any $u\in T$.
Otherwise, there exist some $v_i$ such that $u\sim v_i$, where $i\geq a-n_2+1.$ Note that $d_{Q_1}(u)=a-n_2.$ Then there must be some $v_j$ such that $u\nsim v_j$, where $1\leq j\leq a-n_2.$
Let $G''=G^*-uv_i+uv_j.$ Note that $G''\in \mathcal{G}_n^{a,b} .$ By Lemma \ref{le2.4}, $\rho(G'')>\rho$, contradicting the maximality of $\rho$. Hence $u$ is adjacent to $v_ 1, v_2, \ldots, v_{a-n_2}$ for any $u\in T$. By symmetry, we have $x_{u_1}=x_{u_2}=\cdots=x_{u_t}$, $x_{v_1}=x_{v_2}=\cdots=x_{v_{a-n_2}}$, and $x_{v_{a-n_2+1}}=x_{v_{a-n_2+2}}=\cdots=x_{v_{n-a-2}}$.

Let $E_1=\{w_iw_j| 2\leq i <j\leq n_2\}\cup \{w_iu_j| 2\leq i\leq n_2, 1\leq j\leq t\}$ and $E_2=\{u_iv_j| 1\leq i\leq t, a-n_2+1\leq j \leq a-1\}\cup \{w_iv_j| 2\leq i\leq n_2, 1\leq j\leq a-1\}$. Define $G^{\star} =G^*-E_1+E_2$, and let $\boldsymbol y$ be the Perron vector of $A(G^{\star})$. Clearly, $G^{\star}\cong G_{n}^{a}\in \mathcal{G}_n^{a,b}$. By symmetry, we have $y_{u_1}=y_{u_i}=y_{w_j}$ for any $2\leq i\leq t, 2\leq j \leq n_2$, $y_{v_1}=y_{v_k}$ for any $2\leq k\leq a-1$, and $y_{v_a}=y_{v_l}$ for any $a+1\leq l \leq n-a-2$. Let $\rho^{\star}=\rho(G^{\star})$. Clearly, $\rho^{\star} > n-a-3$. By $A(G^{\star})\boldsymbol y=\rho^{\star} \boldsymbol y$, we have
\begin{eqnarray}
\rho^{\star}y_{w_1}&=&(a+1)y_{u_1},\label{eq11}\\
\rho^{\star}y_{v_{n-a-2}}&=&(a-1)y_{v_1}+(n-2a-2)y_{v_{n-a-2}},\label{eq12}\\
\rho^{\star}y_{u_1}&=&(a-1)y_{v_1}+y_{w_1}.\label{eq13}
\end{eqnarray}
 According to \eqref{eq11}, $\rho^{\star} > n-a-3$, and $n \geq 2a^2 + 136a + 164$, we have  $y_{w_1}=\frac{a+1}{\rho^{\star}}y_{u_1}<\frac{a+1}{n-a-3}y_{u_1}<y_{u_1}$.
By \eqref{eq12}, $y_{v_{n-a-2}}=\frac{a-1}{\rho^{\star}-n+2a+2}y_{v_1}<y_{v_1}$. Combining \eqref{eq13} and $y_{w_1}<y_{u_1}$, we have $y_{v_{n-a-2}}=\frac{a-1}{\rho^{\star}-n+2a+2}y_{v_1}>\frac{a-1}{\rho^{\star}-1}y_{v_1}>y_{u_1}$.
Hence $y_{w_1}<y_{u_1}<y_{v_{n-a-2}}$ and $y_{v_1}=y_{v_2}=\cdots=y_{v_{a-1}}> y_{v_a}=\cdots = y_{v_{n-a-2}}$. Recall that $3\leq t\leq a-1$ and $3\leq n_2 \leq a-1$. Then
\begin{eqnarray*}
&&\boldsymbol{y}^{T}(\rho^{\star}-\rho)\boldsymbol{x}\\
&=& \boldsymbol{y}^{T}(A(G^{\star})-A(G^*))\boldsymbol{x}\\
&=&\sum_{u_iv_j\in E_2}(x_{u_i}y_{v_j}+x_{v_j}y_{u_i})-\sum_{u_iv_j\in E_1}(x_{u_i}y_{v_j}+x_{v_j}y_{u_i})\\
&\geq& [t(n_2-1)+(a-1)(n_2-1)](x_{w_1}y_{v_1}+x_{v_{a-1}}y_{u_1})\\
&& -\left[\frac{(n_2-1)(n_2-2)}{2}+t(n_2-1)\right](x_{w_1}y_{u_1}+x_{u_1}y_{u_1})\\
&>& t(n_2-1)(x_{w_1}y_{v_1}+x_{v_{a-1}}y_{u_1}-x_{w_1}y_{u_1}-x_{u_1}y_{u_1}\\
&& +(n_2-1)\left[(a-1)-\frac{n_2-2}{2}\right](x_{w_1}y_{v_1}+x_{v_{a-1}}y_{u_1})\\
&\geq&
t(n_2-1)[x_{w_1}(y_{v_1}-y_{u_1})+y_{u_1}(x_{v_{a-1}}-x_{u_1})]+(a+1)(x_{w_1}y_{v_1}+x_{v_{a-1}}y_{u_1})\\
&>&0.
\end{eqnarray*}
It follows that $\rho(G^{\star})>\rho$, which contradicts the maximality of $\rho$. Hence $t=a+1.$
\end{proof}

Now we are in a position to present the proof of Theorem \ref{thm1.3}.

\medskip
\noindent  \textbf{Proof of Theorem \ref{thm1.3}.}
By Lemma \ref{le3.9}, $t=a+1$ and $n_2=1.$
Combining Lemma \ref{le3.4}, we have $|[V(Q_2),T]|_{G^*} = a+1$. Let $T = \{u_1, u_2, \dots, u_{a+1}\}$, $V(Q_1) = \{v_1, v_2, \dots, v_{n-a-2}\}$, and $V(Q_2) = \{w_1\}$. Without loss of generality, assume that $x_{u_1} \geq x_{u_2} \geq \cdots \geq x_{u_{a+1}}$ and $x_{v_1} \geq x_{v_2} \geq \cdots \geq x_{v_{n-a-2}}$. By $A(G^*)\boldsymbol{x}=\rho\boldsymbol{x}$, we have
$$\rho x_{w_1}= \sum_{i=1}^{a+1}x_{u_i} \leq (a+1)x_{u_1}.$$
Combining $\rho>n-a-3$ and $n \geq 2a^2 + 136a + 164$, we can deduce that
\begin{eqnarray*}
x_{w_1}\leq \frac{a+1}{\rho}x_{u_1} < \frac{a+1}{n-a-3}x_{u_1}< x_{u_1}.
\end{eqnarray*}
Let $d_T(u_1)=r$. Combining $x_{w_1}<x_{u_1}$ and $d_{G^*}(u_1)=a$, we have

\begin{eqnarray*}
\rho x_{u_1} &=& \sum_{u\in T, u\sim u_1} x_u + \sum_{v\in V(Q_1), v\sim u_1} x_v + x_{w_1}\leq r x_{u_1} + \sum_{1\leq i\leq a-r-1} x_{v_i} + x_{w_1}\\
&<&(r+1) x_{u_1} + \sum_{1\leq i\leq a-r-1} x_{v_i},
\end{eqnarray*}
and hence
\begin{eqnarray*}
x_{u_1} < \frac{1}{\rho - r - 1} \sum_{1\leq i\leq a-r-1} x_{v_i}.
\end{eqnarray*}
Since $x_{u_1}>0$, $a-r-1\geq 1$, which implies that $r\leq a-2$. Note that
\begin{eqnarray*}
\rho x_{v_{n-a-2}} \geq \sum_{v\in V(Q_1), v\sim v_{n-a-2}} x_v \geq \sum_{1\leq i\leq a-r-1} x_{v_i} + (n-2a+r-2) x_{v_{n-a-2}}.
\end{eqnarray*}
Then we have
\begin{eqnarray*}
x_{v_{n-a-2}} \geq \frac{1}{\rho - (n-2a-3)-r-1} \sum_{1\leq i\leq a-r-1} x_{v_i}.
\end{eqnarray*}
By $\rho>n-a-3$, $r\leq a-2$, and $n \geq 2a^2 + 136a + 164$, we obtain that $1 < \rho - (n-2a-3)-r-1  < \rho - r-1$. Hence
\begin{eqnarray}\label{eq14}
x_{v_{n-a-2}}\geq \frac{1}{\rho - (n-2a-3)-r-1} \sum_{1\leq i\leq a-r-1} x_{v_i}> \frac{1}{\rho - r - 1} \sum_{1\leq i\leq a-r-1} x_{v_i}>x_{u_1}.
\end{eqnarray}

We first claim that $E(G^*[T])= \emptyset$. Otherwise, there exists an edge $uu'\in E(G^*[T])$. Note that $|N_{G^*}(u)\cup N_{G^*}(u')|\leq 2a <|V(Q_1)|$. Then there exists some vertex $v\in V(Q_1)$ such that $vu, vu' \notin E(G^*)$. We construct a new graph $G'$ from $G^*$ by deleting the edge $uu'$ and adding edges $vu$ and $vu'$.
One can observe that $G'\in \mathcal{G}_n^{a,b}$. By \eqref{eq14} and Lemma \ref{le2.4}, $\rho(G')>\rho,$ contradicting the maximality of $\rho$. Hence there are no edges inside $T$.

Finally, we can prove that $u$ must be adjacent to $v_ 1, v_2, \ldots, v_{a-1}$ for any $u\in T$.
Otherwise, there exists some $v_i$ such that $u\sim v_i$, where $i\geq a.$ Note that $d_{Q_1}(u)=a-1.$ Then there must exist some $v_j$ such that $u\nsim v_j$, where $1\leq j\leq a-1.$
Let $G''=G^*-uv_i+uv_j.$ Note that $G''\in  \mathcal{G}_n^{a,b} $ and $x_{v_j}\geq x_{v_i}$. By Lemma \ref{le2.4}, $\rho(G'')>\rho$, which contradicts the maximality of $\rho$. Hence $G^*\cong G_{n}^{a}.$
\hspace*{\fill}$\Box$

\vspace{10mm}
\noindent
{\bf Declaration of competing interest}
\vspace{3mm}

The authors declare that they have no known competing financial interests or personal relationships that could have appeared to influence the work reported in this paper.

\vspace{5mm}
\noindent
{\bf Data availability}
\vspace{3mm}

No data was used for the research described in this paper.



\end{document}